\documentclass[12pt]{amsart}
\usepackage{mathrsfs}
\usepackage{amsfonts}
\usepackage{enumitem}
\usepackage{amsmath, amsfonts, amssymb, amsthm, amscd}
\usepackage[all]{xy}
\begin{document}
\theoremstyle{plain}
\newtheorem{thm}{Theorem}[section]
\newtheorem{theorem}[thm]{Theorem}
\newtheorem{lemma}[thm]{Lemma}
\newtheorem{corollary}[thm]{Corollary}
\newtheorem{proposition}[thm]{Proposition}
\newtheorem{conjecture}[thm]{Conjecture}
\newtheorem{definition}[thm]{Definition}
\newtheorem{acknowledgement}[thm]{Acknowledgement}
\theoremstyle{remark}
\newtheorem{construction}[thm]{Construction}
\newtheorem{notations}[thm]{Notations}
\newtheorem{question}[thm]{Question}
\newtheorem{problem}[thm]{Problem}
\newtheorem{remark}[thm]{Remark}
\newtheorem{claim}[thm]{Claim}
\newtheorem{assumption}[thm]{Assumption}
\newtheorem{assumptions}[thm]{Assumptions}
\newtheorem{properties}[thm]{Properties}
\newtheorem{example}[thm]{Example}
\newtheorem{comments}[thm]{Comments}
\newtheorem{blank}[thm]{}
\newtheorem{observation}[thm]{Observation}
\newtheorem{defn-thm}[thm]{Definition-Theorem}
%%%%%%%%%%%%%%%%%%%%%%%%%%%%%%%%%%%%%%%%%%%%%%%%%%%%%%%%%%%%%%%%%%%
\renewcommand{\bar}{\overline}
\newcommand{\eps}{\varepsilon}
\newcommand{\pa}{\partial}
\renewcommand{\phi}{\varphi}
\newcommand{\wt}{\widetilde}
\newcommand{\oo}{\mathcal O}
\newcommand{\ka}{K\"ahler }
\newcommand{\kar}{K\"ahler-Ricci}
\newcommand{\C}{{\mathbb C}}
\newcommand{\Z}{{\mathbb Z}}
\newcommand{\R}{{\mathbb R}}
\newcommand{\N}{{\mathbb N}}
\newcommand{\Q}{{\mathbb Q}}
\newcommand{\pz}{\partial_z}
\newcommand{\pzb}{\partial_{\bar z}}
\newcommand{\M}{{\mathcal M}}
\newcommand{\T}{{\mathcal T}}
\newcommand{\NN}{{\mathcal N}}
\newcommand{\XX}{{\widetilde{\mathfrak X}}}
\newcommand{\pzs}{\partial_{z_s}}
\newcommand{\pzbs}{\partial_{\bar{z_s}}}
\newcommand{\pp}{{\mathbb P}}
\newcommand{\ke}{K\"ahler-Einstein }
\newcommand{\hh}{{\mathcal H}}
\newcommand{\kk}{{\mathcal K}}
\newcommand{\tei}{Teichm\"uller }
\newcommand{\h}{{\mathbb H}}
\renewcommand{\tilde}{\widetilde}
\newcommand{\p}{{\Phi}}
\newcommand{\g}{{\mathfrak g}}
\newcommand{\kkk}{{\mathfrak k}}
\newcommand{\mkp}{{\mathfrak p}}
\newcommand{\lb}{\left (}
\newcommand{\rb}{\right )}
%%%%%%%%%%%%%%%%%%%%%%%%%%%%%%%%%%%%%%%%%%%%%%%%%%%%%%%%%%%%%%%%%%%%%%%%%%%%%%%%%%%%%%%%%%%%%%%%%%%%%%%%%%%%%%%%%

\def\dW{\mbox{diff\:}\times \mbox{Weyl\:}}
\def\End{\operatorname{End}}
\def\Hom{\operatorname{Hom}}
\def\Aut{\operatorname{Aut}}
\def\Diff{\operatorname{Diff}}
\def\im{\operatorname{im}}
\def\tr{\operatorname{tr}}
\def\Pr{\operatorname{Pr}}
\def\Z{\bf Z}
\def\O{\mathcal{O}}
\def\CP{\mathbb{C}\mathbb{P}}
\def\P{\bf P}
\def\Q{\bf Q}
\def\R{\bf R}
\def\C{\mathbb{C}}
\def\H{\bf H}
\def\Hil{\mathcal{H}}
\def\proj{\operatorname{proj}}
\def\id{\mbox{id\:}}
\def\a{\alpha}
\def\b{\beta}
\def\c{\gamma}
\def\p{\partial}
\def\f{\frac}
\def\i{\sqrt{-1}}
\def\t{\tau}
\def\T{\mathcal{T}}
\def\Kahler{K\"{a}hler\:}
\def\w{\omega}
\def\X{\mathfrak{X}}
\def\K{\mathcal {K}}
\def\m{\mu}
\def\M{\mathcal {M}}
\def\v{\nu}
\def\D{\mathcal{D}}
\def\U{\mathcal {U}}
%%%%%%%%%%%%%%%%%%%%%%%%%%%%%%%%%%%%%%%%%%%%%%%%%%%%%%%%%%%%%%%%%%%%%%%%%%%%%%%%%%%%%%
\def\Omegak{\frac{1}{k!}\bigwedge\limits^k\mu\lrcorner\Omega}
\def\Omegakp{\frac{1}{(k+1)!}\bigwedge\limits^{k+1}\mu\lrcorner\Omega}
\def\Omegakpp{\frac{1}{(k+2)!}\bigwedge\limits^{k+2}\mu\lrcorner\Omega}
\def\Omegakm{\frac{1}{(k-1)!}\bigwedge\limits^{k-1}\mu\lrcorner\Omega}
\def\Omegakmm{\frac{1}{(k-2)!}\bigwedge\limits^{k-2}\mu\lrcorner\Omega}
\def\Omegakk{\Omega_{i_1,i_2,\cdots,i_k}}
\def\Omegakkp{\Omega_{i_1,i_2,\cdots,i_{k+1}}}
\def\Omegakkpp{\Omega_{i_1,i_2,\cdots,i_{k+2}}}
\def\Omegakkm{\Omega_{i_1,i_2,\cdots,i_{k-1}}}
\def\Omegakkmm{\Omega_{i_1,i_2,\cdots,i_{k-2}}}
\def\mukm{\frac{1}{(k-1)!}\bigwedge\limits^{k-1}\mu}
\def\sumk{\sum\limits_{i_1<i_2<,\cdots,<i_k}}
\def\sumkm{\sum\limits_{i_1<i_2<,\cdots,<i_{k-1}}}
\def\sumkmm{\sum\limits_{i_1<i_2<,\cdots,<i_{k-2}}}
\def\sumkp{\sum\limits_{i_1<i_2<,\cdots,<i_{k+1}}}
\def\sumkpp{\sum\limits_{i_1<i_2<,\cdots,<i_{k+2}}}
\def\Omegakb{\Omega_{i_1,\cdots,\bar{i}_t,\cdots,i_k}}
\def\Omegakmb{\Omega_{i_1,\cdots,\bar{i}_t,\cdots,i_{k-1}}}
\def\Omegakpb{\Omega_{i_1,\cdots,\bar{i}_t,\cdots,i_{k+1}}}
\def\Omegakt{\Omega_{i_1,\cdots,\tilde{i}_t,\cdots,i_k}}

%%%%%%%%%%%%%%%%%%%%%%%%%%%%%%%%%%%%%%%%%%%%%%%%%%%%%%%%%%%%%%
\title{Applications of the affine structures on the Teichm\"uller spaces}

                      \author{Kefeng Liu}
                \address{Center of Mathematical Sciences, Zhejiang University, Hangzhou, Zhejiang 310027, China;
                 Department of Mathematics,University of California at Los Angeles,
                Los Angeles, CA 90095-1555, USA}
        \email{liu@math.ucla.edu, liu@cms.zju.edu.cn}
         \author{Yang Shen}
\address{Center of Mathematical Sciences, Zhejiang University, Hangzhou, Zhejiang 310027, China}
\email{syliuguang2007@163.com}

\author{Xiaojing Chen}
          \dedicatory{}
        \address{Department of Mathematics,University of California at Los Angeles,
               Los Angeles, CA 90095-1555, USA}
        \email{xjchen@math.ucla.edu}

 \begin{abstract} We prove the existence of global sections trivializing the Hodge bundles on the Hodge metric completion space of the Torelli space of Calabi--Yau manifolds, a global splitting property of these Hodge bundles. We also prove that a compact Calabi--Yau manifold can not be deformed to its complex conjugate. These results answer certain open questions in the subject. A general result about certain period map to be bi-holomorphic from the Hodge metric completion space of the Torelli space of Calabi--Yau type manifolds to their period domains is proved and applied to the cases of K$3$ surfaces, cubic fourfolds, and hyperk\"ahler manifolds.

\end{abstract}
\maketitle
\section*{Introduction}
\label{sec:Introduction}

%This paper is motivated by the certain open questions mentioned in the above abstract brought to us by Professors Bong Lian and Si Li.
As applications of our recent results about the affine structure on the Teichm\"uller spaces of Calabi--Yau manifolds, we answer certain open questions mentioned in the above abstract brought to us by Professors Bong Lian and Si Li affirmatively.

In our recent papers \cite{CGL, CGL2}, we studied the period maps and the Hodge metric completion spaces on the Torelli spaces of Calabi--Yau and Calabi--Yau type manifolds, respectively. 
In \cite{CGL, CGL2}, we have proved global Torelli theorems on the Torelli space of Calabi--Yau manifolds and Calabi--Yau type manifolds. Our method works without change for more general Calabi--Yau manifolds such as hyperk\"ahler manifolds, as long as their moduli spaces with certain level structure are smooth. 
%In \cite{CGL2}, we adapted similar argument from \cite{CGL} to give a proof of the global Torelli theorem for  the Teichm\"uller space of polarized and marked Calabi--Yau type manifolds, which is a more general kind of polarized and marked projective manifolds, including cubic fourfolds. 
Such Torelli problem has been studied for a long time. One may refer to \cite{CGL3} for a brief summary of  the history of Torelli problem. 
%In \cite{aw1} Weil reformulated the Torelli problem for Riemann surfaces with polarization. Andreotti proved Weil's version of the Torelli problem in \cite{and}. In 1960s and 70s, Griffiths \cite{Griffiths1, Griffiths4} and Deligne \cite{Deligne} developed the general theory of variations of Hodge structures, which re-shaped the theory of Torelli problems in terms of period domains and period maps. The global Torelli problem for K$3$ surfaces was conjectured by Weil in \cite{W2} and the proof of it was given by Shafarevich and Piatetski-Shapiro in \cite{PS}, Looijenga in \cite{Looijenga1}, and Burns-Rapoport in \cite{BR}. Todorov \cite{Tod} and Siu \cite{Siu} were able to show the surjectivity of the period mapping for K$3$ surfaces. Furthermore, Voisin in \cite{VoisinCF} and Looijenga in \cite{Looijenga2} proved the global Torelli theorem for moduli spaces of cubic fourfolds, the moduli space and period maps of cubic fourfolds are also well studied in Laza \cite{Laza1, Laza2}. 
%A version of global Torelli theorem for marked hyperk\"ahler manifolds was recently proved by Verbitsky \cite{ver}. Surjectivity of the period map from certain moduli spaces for hyperk\"ahler manifolds was proved by Huybrechts and Verbitsky in \cite{Huy}, \cite{ver}. 
%There are also many works on local Torelli theorems and generic Torelli theorems.

The main object we are considering in this paper is Calabi--Yau manifold. In this paper, a compact projective manifold $M$ of complex dimension $n\geq 2$ is called \textit{Calabi--Yau}, if it has a trivial canonical bundle and satisfies $H^i(M,\mathcal{O}_M)=0$ for $0<i<n$. Although we mainly focus on the study of the moduli spaces of polarized and marked Calabi--Yau manifolds, we remark that our method in this paper works without change for more general Calabi--Yau manifolds such as hyperk\"ahler manifolds, as well as Calabi--Yau type manifolds such as cubic fourfolds. 
%as long as their moduli spaces with certain level structure are smooth. 
%In \cite{CGL2}, we adapted similar argument from \cite{CGL} to give a proof of the global Torelli theorem for  the Teichm\"uller space of polarized and marked Calabi--Yau type manifolds, which is a more general kind of polarized and marked projective manifolds, including cubic fourfolds.

The key property we proved in \cite{CGL} will be recalled in Theorem \ref{image} in $\S$2 in this paper. Based on this theorem, we were able to construct the holomorphic affine structure on the Teichm\"uller spaces as well as on  the Hodge metric completion space of Torelli space, which is a connected component of the moduli space of polarized and marked  Calabi-Yau manifolds, and thus prove the injectivity of the period map on the Torelli space and on its Hodge metric completion space. We will give a brief review of the definitions of basic concepts in $\S$1, and give a sketch of the construction of holomorphic affine structures on the Teichm\"uller spaces in $\S$2. In $\S$3, we first describe the construction of the Hodge metric completion spaces of the Torelli space, and conclude with a global Torelli theorem on the Torelli space. One may refer to \cite{CGL} for more details.

The main purpose of  this paper is to use our results to study period maps and Hodge bundles on the Torelli space of Calabi--Yau manifolds and its Hodge metric completion space.

In Section \ref{surjectivity}, we will first prove a simple but general result in Theorem \ref{surj} concerning when the extended period map is a bi-holomorphic map, and apply it to give a simple proof of the surjectivity of the extended period map on Hodge metric completion of the Torelli space of K$3$ surfaces and cubic fourfolds. Such surjectivity results of the extended period maps may be obtained for the case of  hyperk\"ahler manifolds with weight 2 primitive cohomology, as well as certain refined period maps on cubic threefolds. 

%More precisely, let $\mathcal{T}^H$ denote the Hodge metric completion space of the Teichm\"uller space of a polarized and marked Calabi--Yau manifold, and $D$ be its corresponding period domain. In \cite{CGL} we have proved that $\mathcal{T}^H$ is smooth complex affine manifold if $\mathcal{T}$ is smooth. Then we prove the following theorem for polarized and marked Calabi--Yau,
%
%\begin{theorem}  If $\dim\mathcal{T}^H=\dim D$, then the extended period map $\Phi^H: \mathcal{T}^H\to D$ is surjective.
%\end{theorem}

% Note that in the case of Calabi--Yau type manifolds, we require that the Teichm\"uller spaces are smooth, as in the cases of Calabi--Yau manifolds. 
%In particular, the above theorem applies to the cases of K$3$ and cubic fourfolds, and more generally to hyperk\"ahler manifolds. We note that similar results for the moduli spaces of K$3$ surfaces and cubic fourfolds were first proved in \cite{Siu, Tod} and \cite{Looijenga2} respectively.
%\newline 
In Section \ref{globalSection}, we let $F^k$ with $0\leq k\leq n$ denote the Hodge bundles on the Hodge metric completion of the Torellli space $\mathcal{T}^H$, and prove that all the Hodge bundles over $\mathcal{T}^H$ for Calabi--Yau manifolds are trivial bundles, and the trivialization can be obtained by the canonical sections in \eqref{sections}.  These canonical sections are explicitly constructed in Section \ref{globalSection} by using unipotent matrices in our proof of the affine structure on $\mathcal{T}^H$.
% \begin{theorem} 
%All the Hodge bundles $F^k$ over $\mathcal{T}^H$ are trivial bundles, and the trivialization can be obtained by the canonical sections in \eqref{sections}.
%\end{theorem}
Moreover, let $p\in\mathcal{T}^{H}$ be a base point with the corresponding Calabi--Yau manifolds $M_p$. Let $\phi(\tau)$ denote the Kuranishi family of Beltrami differential for the local deformation of complex structure of a Calabi--Yau manifold $M_p$, and let $[\Omega^c_p(\tau)]=[e^{\phi(\tau)}\lrcorner\Omega_p]$ denote the local canonical family of holomorphic $(n,0)$-classes around the base point $p$ in $\mathcal{T}^{H}$. Then we prove that
there is a global holomorphic section $s_0$ of the Hodge bundle $F^n$ on $\mathcal{T}^H$ which extends the local canonical holomorphic section $[\Omega^c_p(\tau)]$.

In Section \ref{split}, we still consider the Hodge metric completion space of the Torelli space $\mathcal{T}^{H}$ of Calabi--Yau manifolds and let $p\in\mathcal{T}^{H}$ be the base point. Let $F^k_q$ denote the fiber of the Hodge bundle $F^k$ at $q$ for any point $q\in \mathcal{T}^H$. We show that  for any two points $p$ and $q$ in $\mathcal{T}^H$ and $1\leq k\leq n$, we have that $H^n(M, \mathbb{C})=F^k_p\oplus \bar{{F}^{n-k+1}_q}$.

Finally we let $M_q$ denote a compact polarized and marked Calabi--Yau manifold with complex structure $J$, and $\bar{M_{q}}$ denote the corresponding polarized and marked Calabi--Yau manifold with conjugate complex structure $-J$, then we prove that if $q\neq q'$ are two distinct points in the Torelli space $\mathcal{T}'$, then $M_{q'}\neq \bar{M_{q}}$.
Here $M_{q'}\neq \bar{M_{q}}$ means that the two polarized and marked Calabi--Yau manifolds can not be identical in the Torelli space $\mathcal{T}'$. 
Equivalently this implies that $M_{q}$ can not be deformed to its complex conjugate manifold $\bar{M_{q}}$. 
%We remarked such type of problems have been studied, and that one may find results of similar type of problems in \cite{KK}. Theorem \ref{last intro} and Theorem \ref{conjugate intro} answer some long-standing questions in the subject of Calabi--Yau manifolds.
%\begin{theorem}\label{last}
%Vector bundles $\{\tilde{F}^k\}_{k=0}^n$ are globally defined anti-holomorphic vector bundles over $\mathcal{T}^H$ such that
%\begin{align*}
%H^n(M, \mathbb{C})=F^k_p\oplus \tilde{F}^k_q\quad \text{for any } q\in \mathcal{T}^H.
%\end{align*}\\
%\end{theorem}

%we will prove two more interesting results in Section: a result about the existence of a global section on the Hodge bundles of Teichm\"uller spaces, and a result about a global splitting property of the Hodge bundles.

We would like to thank Professors Si Li and Bong Lian for stimulating discussions and many helpful suggestions.

\section{Period maps on moduli spaces}

This section is a review of basic definitions of Calabi--Yau manifolds, as well as their moduli, Tecihm\"uller, and Torelli spaces. Then we study the period maps on them. 
Basic results about the moduli, Tecihm\"uller, and Torelli spaces of Calabi--Yau manifolds are briefly recalled in Section 1.1. In Section 1.2 basic properties of period maps are also reviewed for reader's convenience.

\subsection{Moduli, Torelli and Teichm\"{u}ller space}\label{moduli and teichmuller}
A compact projective manifold $M$ of complex dimension $n\geq 2$ is called \textit{Calabi--Yau}, if it has a trivial canonical bundle and satisfies $H^i(M,\mathcal{O}_M)=0$ for $0<i<n$. Let $M$ be a Calabi--Yau manifold with $\dim_{\mathbb{C}}M=n$. Let $L$ be an ample line bundle over $M$. We call the pair $(M,L)$ a polarized projective manifold.

The \textit{moduli space} $\mathcal{M}$ of polarized complex structures on a given differential manifold $X$ is a complex analytic space consisting of biholomorphically equivalent pairs $(M, L)$ of complex structures and ample line bundles. Let us denote by $[M, L]$ the point in $\mathcal{M}$ corresponding to a pair $(M, L)$, where $M$ is a complex manifold diffeomorphic to $X$ and $L$ is an ample line bundle on $M$. If there is a biholomorphic map $f$ between $M$ and $M'$ with $f^*L'=L$, then $[M, L]=[M', L']\in\mathcal{M}$. One may also look at \cite{Popp} and \cite{v1} for more details about the construction of moduli spaces of polarized projective manifolds.

Let $(M,L)$ be a polarized Calabi--Yau manifold.  
We fix a lattice $\Lambda$ with a pairing $Q_{0}$, where $\Lambda$ is isomorphic to $H^n(M_{0},\mathbb{Z})/\text{Tor}$ for some Calabi--Yau manifold $M_{0}$ and $Q_{0}$ is defined by the cup-product.
For a polarized Calabi--Yau manifold $(M,L)$, we define a marking $\gamma$ as an isometry of the lattices
\begin{align}\label{marking}
\gamma :\, (\Lambda, Q_{0})\to (H^n(M,\mathbb{Z})/\text{Tor},Q).
\end{align}
A \textit{polarized and marked} projective manifold is a triple $(M,L,\gamma)$ consisting of a projective manifold $M$, an ample line bundle $L$ over $M$, and a marking $\gamma$.

For any integer $m\geq 3$, we follow the definition of Szendr\"oi \cite{sz} to define an $m$-equivalent relation of two markings on $(M,L)$ by 
$$\gamma\sim_{m} \gamma' \text{ if and only if } \gamma'\circ \gamma^{-1}-\text{Id}\in m \cdot\text{End}(H^n(M,\mathbb{Z})/\text{Tor},Q),$$
and denote $[\gamma]_{m}$ to be the set of the $m$-equivalent classes of $\gamma$.
Then we call $[\gamma]_{m}$ a level $m$
structure on the polarized Calabi--Yau manifold $(M,L)$.
For deformation of
polarized Calabi--Yau manifold with level $m$ structure, we reformulate Theorem~2.2 in
\cite{sz} to the following theorem, in which we only put the statements we need in this paper. One can
also look at \cite{Popp} and \cite{v1} for more details about the
construction of moduli spaces of Calabi--Yau manifolds.

\begin{theorem}\label{Szendroi theorem 2.2}
Let $M$ be a polarized Calabi--Yau manifold with level
$m$ structure with $m\geq 3$. Then there exists a connected quasi-projective complex manifold
$\mathcal{Z}_m$ with a universal  family of Calabi--Yau manifolds,
%\begin{align}\label{szendroi universal  family}
$\mathcal{X}_{\mathcal{Z}_m}\rightarrow \mathcal{Z}_m,$
%end{align}
which contains $M$ as a fiber and is polarized by an ample line bundle
$\mathcal{L}_{\mathcal{Z}_m}$ on $\mathcal{X}_{\mathcal{Z}_m}$.
\end{theorem}

For $m\ge 3$, let $\T_{m}$ be the universal covering space of $\mathcal{Z}_{m}$.
It is proved in \cite[$\S$2]{CGL} that $\T_{m}$ is bi-holomorphic to $\T_{m'}$ for any integer $m,m'\ge 3$.
Hence we denote by $\T$ the universal covering space of $\mathcal{Z}_{m}$ for any $m\ge 3$ and call $\T$ the Teichm\"uller space of Calabi--Yau manifolds.
Let $\pi_m:\, \T \to \mathcal{Z}_{m}$ be the universal covering map. 
We then have the pull-back family $\phi:\, \mathcal{U}\rightarrow \mathcal{T}$ of $\mathcal{X}_{\mathcal{Z}_m}\rightarrow \mathcal{Z}_m$ via the covering map $\pi_{m}$.

\begin{proposition}\label{imp}
The Teichm\"uller space $\mathcal{T}$ is a connected and simply connected smooth complex
manifold and the family
%\begin{align}\label{universal  family over Teich}
$\phi:\,\mathcal{U}\rightarrow\mathcal{T},$
%\end{align}
which contains $M$ as a fiber, is universal at each point of
$\mathcal{T}$.
\end{proposition}

We now define the \textit{Torelli space} $\mathcal{T}'$ to be a component of the  complex analytic space consisting of biholomorphic equivalent triples of $(M, L, \gamma)$. To be more precise, 
for two triples $(M, L, \gamma)$ and $(M', L', \gamma'\})$, if there exists a biholomorphic map $f:\,M\to M'$ with
\begin{align*}
f^*L'&=L,\quad \quad
f^*\gamma'=\gamma,
\end{align*} 
where $f^*\gamma'$ is given by $\gamma':\, (\Lambda, Q_{0})\to (H^n(M',\mathbb{Z})/\text{Tor},Q)$ composed with 
$$f^*:\, (H^n(M',\mathbb{Z})/\text{Tor},Q)\to (H^n(M,\mathbb{Z})/\text{Tor},Q),$$
then $[M, L, \gamma]=[M', L',\gamma']\in \mathcal{T}'$. 

More precisely we define the Torelli space as follows.

\begin{definition}
The Torelli space $\T'$ of Calabi--Yau manifolds is the connected component of the moduli space of equivalent classes of polarized and marked Calabi--Yau manifolds, which contains $(M, L)$.
\end{definition}

There is a natural covering map $\pi_m':\,\mathcal{T}'\to\mathcal{Z}_m$ by mapping $(M, L, \gamma)$ to $(M, L, [\gamma]_{m})$. From this we see easily that $\T'$ is a smooth and connected complex manifold.
We can also get a pull-back universal family $\phi':\, \U'\to \T'$ on the Torelli space $\T'$ via the covering map $\pi_m'$.

Recall that the Teichm\"uller space $\T$ is defined to be the universal covering space of $\mathcal{Z}_{m}$ with covering map $\pi_{m}:\, \T\to \mathcal{Z}_{m}$. Then we can lift $\pi_{m}$ via the covering map $\pi_m':\,\mathcal{T}'\to\mathcal{Z}_m$ to get a covering map $\pi:\, \T\to \T'$, such that the following diagram commutes.
\begin{equation}\label{cover1}
\xymatrix{
\T \ar[dr]^-{\pi} \ar[dd]^-{\pi_m} &\\
&\T'\ar[dl]^-{\pi_m'}\\
\mathcal{Z}_{m} &\ .}
\end{equation}

\subsection{Period domain and period map}
We will give a general review about Hodge structures and period domain. One may refer to $\S3$ in \cite{schmid1} for more details.
For a polarized and marked Calabi--Yau manifold $M$ with background
smooth manifold $X$, the marking identifies
$H^n(M,\mathbb{Z})/\text{Tor}$ with a fixed lattice $\Lambda$. This
gives us a canonical identification of the middle dimensional de
Rham cohomology of $M$ to that of the background manifold $X$, that is,
%\begin{equation*}
$H^n(M)\cong H^n(X),$
%\end{equation*}
where the coefficient ring can be ${\mathbb{Q}}$, ${\mathbb{R}}$ or
${\mathbb{C}} $. Since the polarization $L$ is an integer class, it defines a map
%\begin{equation*}
$L:\, H^n(X,{\mathbb{Q}})\to H^{n+2}(X,{\mathbb{Q}}),\quad A\mapsto L\wedge A.$
%\end{equation*}
We denote by $H_{pr}^n(X)=\ker(L)$ the primitive cohomology groups
where, the coefficient ring can also be ${\mathbb{Q}}$, ${\mathbb{R}}$ or ${\mathbb{C}%
}$. We let $H_{pr}^{k,n-k}(M)=H^{k,n-k}(M)%
\cap H_{pr}^n(M,{\mathbb{C}})$ and denote its dimension by
$h^{k,n-k}$. We have the Hodge decomposition
%\begin{align*}  %\label{cl10}
$H_{pr}^n(M,{\mathbb{C}})=H_{pr}^{n,0}(M)\oplus\cdots\oplus
H_{pr}^{0,n}(M)$, such that $H_{pr}^{n-k,k}(M)=\bar{H_{pr}^{k,n-k}(M)}$, $0\le k\le n$.
%\end{align*}
%It is easy to see that for a polarized Calabi--Yau manifold, since $H^2(M, {\mathcal O}_M)=0$, we have
%$$H_{pr}^{n,0}(M)= H^{n,0}(M), \ H_{pr}^{n-1,1}(M)= H^{n-1,1}(M).$$
The Poincar\'e bilinear form $Q$ on $H_{pr}^n(X,{\mathbb{Q}})$ is
defined by
\begin{equation*}
Q(u,v)=(-1)^{\frac{n(n-1)}{2}}\int_X u\wedge v
\end{equation*}
for any $d$-closed $n$-forms $u,v$ on $X$.
Let $f^k=\sum_{i=k}^nh^{i,n-i}$, $f^0=m$, and
$F^k=F^k(M)=H_{pr}^{n,0}(M)\oplus\cdots\oplus H_{pr}^{k,n-k}(M)$,
from which we have the decreasing filtration
$H_{pr}^n(M,{\mathbb{C}})=F^0\supset\cdots\supset F^n.$ We know that
\begin{align}
&\dim_{\mathbb{C}} F^k=f^k,  \label{cl45}\\
& H^n_{pr}(X,{\mathbb{C}})=F^{k}\oplus \bar{F^{n-k+1}}, \text{ and }
H_{pr}^{k,n-k}(M)=F^k\cap\bar{F^{n-k}}.\label{cl46}
\end{align}
In terms of the Hodge filtration, the Poincar\'e bilinear form satisfies the Hodge-Riemann relations
\begin{align}
&Q\left ( F^k,F^{n-k+1}\right )=0, \quad\text{and}\quad\label{cl50}\\
&Q\left ( Cv,\bar v\right )>0 \quad\text{if}\quad v\ne 0,\label{cl60}
\end{align}
where $C$ is the Weil operator given by $Cv=\left (\sqrt{-1}\right
)^{2k-n}v$ for $v\in H_{pr}^{k,n-k}(M)$. The period domain $D$
for polarized Hodge structures with data \eqref{cl45} is the space
of all such Hodge filtrations
%\begin{equation*}
$D=\left \{ F^n\subset\cdots\subset F^0=H_{pr}^n(X,{\mathbb{C}})\mid %
\eqref{cl45}, \eqref{cl50} \text{ and } \eqref{cl60} \text{ hold}
\right \}.$
%\end{equation*}
The compact dual $\check D$ of $D$ is
%\begin{equation*}
$\check D=\left \{ F^n\subset\cdots\subset F^0=H_{pr}^n(X,{\mathbb{C}})\mid %
\eqref{cl45} \text{ and } \eqref{cl50} \text{ hold} \right \}.$
%\end{equation*}
The period domain $D\subseteq \check D$ is an open subset. 
%We note that the conditions \eqref{cl50} and \eqref{cl60} imply the identities % in \eqref{cl46}.
\begin{remark}We remark the notation change for the primitive cohomology groups. For simplicity, we will use $H^n(M,\mathbb{C})$ and $H^{k, n-k}(M)$ to denote the primitive cohomology groups $H^n_{pr}(M,\mathbb{C})$ and $H_{pr}^{k, n-k}(M)$ respectively. Moreover, we will use cohomology to mean primitive cohomology in the rest of the paper.
\end{remark}
%\begin{remark}We remark the notation change for the primitive cohomology groups.
%As mentioned above that for a polarized Calabi--Yau manifold,
%$$H_{pr}^{n,0}(M)= H^{n,0}(M), \ H_{pr}^{n-1,1}(M)= H^{n-1,1}(M).$$
%For the reason that we mainly consider these two types of primitive cohomology group of a Calabi--Yau manifold, by abuse of notation, we will simply use $H^n(M,\mathbb{C})$ and $H^{k, n-k}(M)$ to denote the primitive cohomology groups $H^n_{pr}(M,\mathbb{C})$ and $H_{pr}^{k, n-k}(M)$ respectively. Moreover, we will use cohomology to mean primitive cohomology in the rest of the paper.
%\end{remark}

For the family $f_m : \mathcal{X}_{\mathcal{Z}_{m}} \to \mathcal{Z}_{m}$, we denote each fiber by $M_s=f_m^{-1}(s)$ and $F_s^k=F^k(M_s)$ for any $s\in \mathcal{Z}_{m}$. With some fixed point $s_0\in \mathcal{Z}_m$, the period map is defined as a morphism $\Phi_{\mathcal{Z}_m} :\mathcal{Z}_m \to D/\Gamma$ by
\begin{equation}\label{perioddefinition}
s \mapsto \tau^{[\gamma_s]}(F^n_s\subseteq\cdots\subseteq F^0_s)\in D,
\end{equation}
where $\tau^{[\gamma_s]}$ is an isomorphism between $\C-$vector spaces
$$\tau^{[\gamma_s]}:\, H^n(M_s,\mathbb{C})\to H^n(M_{s_0},\mathbb{C}),$$
which depends only on the homotopy class $[\gamma_s]$ of the curve $\gamma_s$ between $s$ and $s_0$. Then the period map is well-defined with respect to the monodromy representation 
$\rho : \pi_1(\mathcal{Z}_m)\to \Gamma \subseteq \text{Aut}(H_{\mathbb{Z}},Q)$. It is well-known that the period map has the following properties:
\begin{enumerate}
\item locally liftable;
\item holomorphic, i.e. $\partial F^i_z/\partial \bar{z}\subset F^i_z$, $0\le i\le n$;
\item Griffiths transversality: $\partial F^i_z/\partial z\subset F^{i-1}_z$, $1\le i\le n$.
\end{enumerate}

Since (1) and $\T$ is the universal covering space of $\mathcal{Z}_{m}$, we can lift the period map to $\Phi : \T \to D$ such that the diagram
$$\xymatrix{
\T \ar[r]^-{\Phi} \ar[d]^-{\pi_m} & D\ar[d]^-{\pi}\\
\mathcal{Z}_m \ar[r]^-{\Phi_{\mathcal{Z}_m}} & D/\Gamma
}$$
is commutative. From the definition of marking in \eqref{marking}, we also have a well-defined period map $\Phi':\, \T'\to D$ from the Torelli space $\T'$ by defining 
\begin{equation*}
p \mapsto \gamma_{p}^{-1}(F^n_p\subseteq\cdots\subseteq F^0_p)\in D,
\end{equation*}
where the triple $(M_{p}, L_{p}, \gamma_{p})$ is the fiber over $p\in \T'$ of the family $\phi':\, \U'\to \T'$. Then we have the following commutative diagram
\begin{align}\label{periods}
\xymatrix{
\T \ar[dr]^-{\pi}\ar[rr]^-{\Phi} \ar[dd]^-{\pi_m} && D\ar[dd]^-{\pi_{D}}\\
&\T'\ar[ur]^-{\Phi'}\ar[dl]_-{\pi'_{m}}&\\
\mathcal{Z}_{m} \ar[rr]^-{\Phi_{\mathcal{Z}_{m}}} && D/\Gamma,
}
\end{align}
where the  maps $\pi_{m}$, $\pi'_{m}$ and $\pi$ are all natural covering maps between the corresponding spaces as in \eqref{cover1}.

%Then the period map
%$\Phi:\,\mathcal{T}\rightarrow D$ is defined by assigning to each point in
%$\mathcal{T}$ the Hodge structure of the corresponding fiber. The period map has several good properties, and one may refer to
%Chapter~10 in \cite{Voisin} for details. Among them, one of the most
%important is the following Griffiths transversality:
%the period map $\Phi$ is a holomorphic map and its tangent map satisfies that
%\begin{equation*}%\label{horizontal}
%\Phi_*(v)\in \bigoplus_{k=1}^{n}
%\text{Hom}\left(F^k_p/F^{k+1}_p,F^{k-1}_p/F^{k}_p\right)\quad\text{for any}\quad p\in\T\ \
%\text{and}\ \ v\in T_p^{1,0}\T
%\end{equation*}
%with $F^{n+1}=0$, or equivalently,
%$\Phi_*(v)\in \bigoplus_{k=1}^{n} \text{Hom} (F^k_p,F^{k-1}_p/F^k_p).$

\section{Affine structures on Teichm\"uller spaces}\label{affine structure}

This section reviews our construction of the affine structure on the Teichm\"uller spaces of Calabi--Yau manifolds.

We will first give a general review about Hodge structures and period domain from Lie group point of view. One may refer to \cite{GS} and \cite{schmid1} for more details.
Let us simply denote $H_{\mathbb{C}}=H^n(M, \mathbb{C})$ and $H_{\mathbb{R}}=H^n(M, \mathbb{R})$. The group of the $\mathbb{C}$-rational points is
%\begin{align*}
$G_{\mathbb{C}}=\{ g\in GL(H_{\mathbb{C}})|~ Q(gu, gv)=Q(u, v) \text{ for all } u, v\in H_{\mathbb{C}}\},$
%\end{align*}
which acts on $\check{D}$ transitively. The group of real points in $G_{\mathbb{C}}$ is
%\begin{align*}
$G_{\mathbb{R}}=\{ g\in GL(H_{\mathbb{R}})|~ Q(gu, gv)=Q(u, v) \text{ for all } u, v\in H_{\mathbb{R}}\},$
%\end{align*}
which acts transitively on $D$ as well. Then we have the following identification,
\begin{align*}
\check{D}\simeq G_\mathbb{C}/B, \ \ \text{ with }\ \  B=\{ g\in G_\mathbb{C}|~ gF^k_p=F^k_p, \text{ for any } k\}.
\end{align*}
Similarly, one obtains an analogous identification
$D\simeq G_\mathbb{R}/V \text{ with }V=G_\mathbb{R}\cap B.$
The Lie algebra $\mathfrak{g}$ of the simple complex Lie group $G_{\mathbb{C}}$ can be described as
\begin{align*}
\mathfrak{g}=\bigoplus_{k\in \mathbb{Z}} \mathfrak{g}^{k, -k}, \text{ with }\mathfrak{g}^{k, -k}=\{X\in \mathfrak{g}|XH_p^{r, n-r}\subseteq H_p^{r+k, n-r-k}\}.
\end{align*}
It is a simple complex Lie algebra, which contains
$\mathfrak{g}_{0}=\{X\in \mathfrak{g}|~ XH_{\mathbb{R}}\subseteq H_\mathbb{R}\}$
as a real form, i.e. $\mathfrak{g}=\mathfrak{g}_0\oplus i \mathfrak{g}_0.$ With the inclusion $G_{\mathbb{R}}\subseteq G_{\mathbb{C}}$, $\mathfrak{g}_0$ becomes Lie algebra of $G_{\mathbb{R}}$.
The Lie algebra $\mathfrak{b}$ of $B$ is then $\mathfrak{b}=\bigoplus_{k\geq 0} \mathfrak{g}^{k, -k}.$
By left translation via $G_{\mathbb{C}}$, $\mathfrak{b}\oplus\mathfrak{g}^{-1,1}/\mathfrak{b}$ gives rise to a $G_{\mathbb{C}}$-invariant holomorphic subbundle of the holomorphic tangent bundle $T^{1,0}\check{D}$. It will be denoted by $T^{1,0}_{h}\check{D}$. The horizontal tangent subbundle, restricted to $D$, determines a  subbundle $T_{h}^{1, 0}D$ of the holomorphic tangent bundle $T^{1, 0}D$ of $D$. In \cite{schmid1}, a holomorphic mapping $\Psi: \,{M}\rightarrow \check{D}$ of a complex manifold $M$ into $\check{D}$ is called \textit{horizontal} if at each point of $M$, the induced map between the holomorphic tangent spaces takes values in the appropriate fibre of $T^{1,0}_{h}\check{D}$. It is not hard to see that the period map $\Phi: \, \mathcal{T}\rightarrow D$ is a horizontal map with $\Phi_*(T^{1,0}_p\mathcal{T})\subseteq T^{1,0}_{o,h}D$ for any $p\in \mathcal{T}$ and $o=\Phi(p)\in D$.

Let us consider the nilpotent Lie subalgebra $\mathfrak{n}_+:=\oplus_{k\geq 1}\mathfrak{g}^{-k,k}$. Then one gets the holomorphic isomorphism $\mathfrak{g}/\mathfrak{b}\cong \mathfrak{n}_+$ which also induces a metric on $\mathfrak{n}_+$. 
Since $D$ is an open set in $\check{D}$, we have the following relation:
%\begin{align*}
$T^{1,0}_{o, h}D= T^{1, 0}_{o, h}\check{D}\cong\mathfrak{b}\oplus \mathfrak{g}^{-1, 1}/\mathfrak{b}\hookrightarrow \mathfrak{g}/\mathfrak{b}\cong \mathfrak{n}_+.$
%\end{align*}

We take the unipotent group $N_+=\exp(\mathfrak{n}_+)$. Then $N_+\simeq \mathfrak{n}_+\simeq \mathbb{C}^d$ for some $d$ with the induced Euclidean metric from $\mathfrak{n}_+$. We remark that with a fixed base point, we can identify $N_+$ with its unipotent orbit in $\check{D}$ by identifying an element $c\in N_+$ with $[c]=cB$ in $\check{D}$; that is, $N_+=N_+(\text{ base point })\cong N_+B/B\subseteq\check{D}.$

%%%%%%%%%%%%%
Let us introduce the notion of an adapted basis for the given Hodge decomposition or the Hodge filtration.
For any $p\in \mathcal{T}$ and $f^k=\dim F^k_p$ for any $0\leq k\leq n$, we call a basis $$\xi=\left\{ \xi_0, \xi_1, \cdots, \xi_N, \cdots, \xi_{f^{k+1}}, \cdots, \xi_{f^k-1}, \cdots,\xi_{f^{2}}, \cdots, \xi_{f^{1}-1}, \xi_{f^{0}-1} \right\}$$ of $H^n(M_p, \mathbb{C})$ an \textit{adapted basis for the given Hodge decomposition}
$H^n(M_p, {\mathbb{C}})=H^{n, 0}_p\oplus H^{n-1, 1}_p\oplus\cdots \oplus H^{1, n-1}_p\oplus H^{0, n}_p, $
if it satisfies $
H^{k, n-k}_p=\text{Span}_{\mathbb{C}}\left\{\xi_{f^{k+1}}, \cdots, \xi_{f^k-1}\right\}$ with $\dim H^{k,n-k}_p=f^k-f^{k+1}$.
%We call a basis
%\begin{align*}
%\zeta=\{\zeta_0, \zeta_1, \cdots, \zeta_N, \cdots, \zeta_{f^{k+1}}, \cdots, \zeta_{f^k-1}, \cdots, \zeta_{f^2}, \cdots, \zeta_{f^1-0}, \zeta_{f^0-1}\}
%\end{align*}
%of $H^n(M_p, \mathbb{C})$ an \textit{adapted basis for the given filtration}
%\begin{align*}
%F^n\subseteq F^{n-1}\subseteq\cdots\subseteq F^0
%\end{align*}
%if it satisfies $F^{k}=\text{Span}_{\mathbb{C}}\{\zeta_0, \cdots, \zeta_{f^k-1}\}$ with $\text{dim}_{\mathbb{C}}F^{k}=f^k$.
Moreover, unless otherwise pointed out, the matrices in this paper are $m\times m$ matrices, where $m=f^0$. The blocks of the $m\times m$ matrix $T$ is set as follows:
for each $0\leq \alpha, \beta\leq n$, the $(\alpha, \beta)$-th block $T^{\alpha, \beta}$ is
\begin{align}\label{block}
T^{\alpha, \beta}=\left[T_{ij}\right]_{f^{-\alpha+n+1}\leq i \leq f^{-\alpha+n}-1, \ f^{-\beta+n+1}\leq j\leq f^{-\beta+n}-1},
\end{align} where $T_{ij}$ is the entries of
the matrix $T$, and $f^{n+1}$ is defined to be zero. In particular, $T =[T^{\alpha,\beta}]$ is called a \textit{block lower triangular matrix} if
$T^{\alpha,\beta}=0$ whenever $\alpha<\beta$.
%%%%%%%%%%%%

Let us fix an adapted basis $\{\eta_0, \cdots, \eta_{m-1}\}$ for the Hodge decomposition of the base point $p\in\mathcal{T}$,
then elements in $N_+$ can be realized as nonsingular block lower triangular matrices whose diagonal blocks are all identity submatrix. By viewing $N_+$ as a subset of $\check{D}$ with the fixed base point, we define
\begin{align*}\check{\mathcal{T}}=\Phi^{-1}(N_+\cap D).
\end{align*}

Let us still denote the restriction map $\Phi|_{_{\check{\mathcal{T}}}}$ by $\Phi$. We first prove the following important proposition by using structure theory for the Lie groups and Lie algebras.
\begin{proposition}\label{bound} The image of the restriction map $\Phi:\,\check{\mathcal{T}}\rightarrow N_+$ is bounded in $N_+\simeq \mathbb{C}^d$ with respect to the Euclidean metric on $N_+$.
\end{proposition}
The main idea of the proof of this theorem is as follows.  First by taking the Cartan decomposition $\mathfrak{g}_{_{\mathbb{R}}}=\mathfrak{k}_0\oplus\mathfrak{p}_0$, we realize that the simple real Lie algebra $\mathfrak{g}_{\mathbb{R}}$ in our case has a Cartan subalgebra $\mathfrak{h}_0\subseteq\mathfrak{k}_0$. Thus $\mathfrak{g}$ is of the first category (cf. \cite{Sugi}). Then using the result of Lemma 3 in \cite{Sugi1} about the real semisimple Lie algebra of first category that the maximal abelian subspace in $\mathfrak{p}_0$ can be decomposed using the noncompact root vectors. This decomposition along with the property that the period map is horizontal in the sense of \cite{schmid1} allows us to show that $\Phi(\check{\mathcal{T}})$ lies in a product of discs in $N_+$. This argument is a slight extension of Harish-Chandra's proof of his famous embedding theorem of the Hermitian symmetric domains as bounded domains in complex Euclidean spaces. One may refer to Lemma 7 and Lemma 8 at pp.~582--583 in \cite{HC}, Proposition 7.4 at pp.~385 and Ch VIII $\S7$ at pp.~382--396 in \cite{Hel}, Proposition 1 at pp.~91 and Proof of Theorem 1 at pp.~95--97 in \cite{Mok}, and Lemma 2.2.12 at pp.~141-142 and $\S5.4$ in \cite{Xu} for more details. Then we apply the Griffiths transversality, together with the geometric structures of integral variations of Hodge structures, to show that the 
image $\Phi(\check{\mathcal{T}})$ lies in a bounded set in $N_+$.

Moreover, since for a generic base point, its orbit of $N_+$ is the largest Schubert cell in $\check{D}$ with complement a divisor, and the period map $\Phi$ is nondegenerate by local Torelli theorm, it is not hard to show that the subset $\check{\mathcal{T}}$ is an open submanifold in $\mathcal{T}$ with the complex codimension of $\mathcal{T}\backslash\check{\mathcal{T}}$ at least one. Then combining with the above boundedness of $\Phi$ on $\check{\mathcal{T}}$, we can apply the Riemann extension theorem to conclude the following theorem.
\begin{theorem}\label{image}The image of the period map $\Phi:\,\mathcal{T}\rightarrow D$ lies in $N_+\cap D$ as a bounded subset.
\end{theorem}
With the above theorem, we have ${\mathcal{T}}=\Phi^{-1}(N_+\cap D).$
Thus for any $q\in {\mathcal{T}}$, $\Phi(q)\in N_+$ is a nonsingular block lower triangular matrix with identity diagonal blocks. 

We consider an Euclidean subspace $A\subset N_{+}$ by $\mathfrak{a}=\Phi_*(\text{T}_p^{1,0}\T)\subset \mathfrak{n}_{+}$ the Abelian subalgebra, and $A=\exp(\mathfrak{a})$ the corresponding Lie group,
and the holomorphic map,
$$\Psi :\, \check{\T} \to A\cap D,$$
where $\Psi=P \circ \Phi|_{\check{\T}}$ and $P$ is the projection map from $N_{+}\cap D$ to $A\cap D$.
%Then we prove the finiteness of the restricted map $P|_{\Phi(\check{\T})}:\, \Phi(\check{\T}) \to A\cap D$.

Now by Theorem \ref{image}, we can extend the holomorphic map $\Psi :\, \check{\T} \to A\cap D$ over $\T$ as
$$\Psi :\, {\T} \to A\cap D,$$
such that $\Psi=P \circ \Phi$.

Let us review the definition of complex affine manifolds.
Let $M$ be a complex manifold of complex dimension $n$. If there
is a coordinate cover $\{(U_i,\,\phi_i);\, i\in I\}$ of M such
that $\phi_{ik}=\phi_i\circ\phi_k^{-1}$ is a holomorphic affine
transformation on $\mathbb{C}^n$ whenever $U_i\cap U_k$ is not
empty, then $\{(U_i,\,\phi_i);\, i\in I\}$ is called a complex
affine coordinate cover on $M$ and it defines a holomorphic affine structure on $M$. One may refer to \cite[pp~ 215]{Mats} for more details.

By using the local Torelli theorem for Calabi--Yau manifolds (cf. \cite{tian1} and \cite{tod1}) and the definition of holomorphic affine structure, we arrive at the main theorem of this section.
%first the holomorphic map $\tau=(\tau_1, \cdots, \tau_N): \,{\mathcal{T}}\rightarrow \mathbb{C}^N$ defines a local coordinate around each point $q\in{\mathcal{T}}$. In particular, the map $\tau$ itself gives a global holomorphic coordinate for ${\mathcal{T}}$. Thus the transition maps are all identity maps. Therefore,
\begin{theorem}The holomorphic map $\Psi :\, {\T} \to A\cap D$ defines a local coordinate around each point $q\in{\mathcal{T}}$. Thus the map $\Psi$ itself gives a global holomorphic coordinate for ${\mathcal{T}}$ with the transition maps all identity maps. In particular, the global holomorphic coordinate $\Psi :\, {\T} \to A\cap D \subset A\simeq \C^{N}$ defines a holomorphic affine structure on ${\mathcal{T}}$. Therefore, ${\mathcal{T}}$ is a complex affine manifold.
\end{theorem}
We remark that the above construction of the holomorphic affine structure can be adopted to polarized and marked Calabi--Yau type manifolds. One may refer to \cite{CGL2} for more details. 

%\begin{remark}We remark that the construction of the holomorphic affine structure on the Teichm\"uller space depends on the choices of the base point. In fact, fix another base point $p'\in\mathcal{T}$, and we analogously define another coordinate map $\tau': \,\mathcal{T}\rightarrow \mathbb{C}^N$, which gives a holomorphic affine structure on $\mathcal{T}$. Then in general $\tau$ and $\tau'$ define different holomorphic affine structures on $\mathcal{T}$ in the following sense: if $\varphi: \mathbb{C}^N\rightarrow \mathbb{C}^N$ is the holomorphic map satisfying $\tau=\phi\circ\tau'$, then $\phi$ is \textit{not} an affine map in general.
%\end{remark}
%\section{Applications}
\section{Hodge metric completion space of Torelli space and a global Torelli theorem}

This section contains a review of our extension of the affine structure to the Hodge metric completion $\mathcal{T}^H$ of the Torelli space $\mathcal{T}'$, as well as its consequences including a global Torelli theorem for polarized and marked Calabi--Yau manifolds.

Let us denote the period map on the smooth moduli space by $\Phi_{\mathcal{Z}_m}:\, \mathcal{Z}_{m}\rightarrow D/\Gamma$, where
$\Gamma$ denotes the global monodromy group which acts
properly and discontinuously on $D$. Then $\Phi: \,\mathcal{T}\to D$ is the lifting of $\Phi_{\mathcal{Z}_m}\circ \pi_m$, where $\pi_m: \,\mathcal{T}\rightarrow \mathcal{Z}_m$ is the universal covering map. There is the Hodge metric $h$ on $D$, which is a complete homogeneous metric and is studied in \cite{GS}. For Calabi--Yau manifolds, both $\Phi_{\mathcal{Z}_m}$ and $\Phi$ are locally injective. Thus the pull-backs of $h$ on $\mathcal{Z}_m$ and $\mathcal{T}$ are both well-defined K\"ahler metrics, similarly one has the pull-back metric on the Toreli space $\mathcal{T}'$.  These metrics are still called the Hodge metrics. 

 By the work of Viehweg in \cite{v1}, we know that $\mathcal{Z}_m$ is quasi-projective and consequently we can find a smooth projective compactification $\bar{\mathcal{Z}}_{m}$ such that $\mathcal{Z}_m$ is Zariski open in $\bar{\mathcal{Z}}_{m}$ and the complement $\bar{\mathcal{Z}}_{m}\backslash\mathcal{Z}_m$ is a divisor of normal crossings. Therefore, $\mathcal{Z}_m$ is dense and open in $\bar{\mathcal{Z}}_{m}$ with the complex codimension of the complement $\bar{\mathcal{Z}}_{m}\backslash \mathcal{Z}_m$ at least one. Moreover as $\bar{\mathcal{Z}}_{m}$ is a compact space, it is a complete space.

Let us now take $\mathcal{Z}^H_{m}$ to be the completion of $\mathcal{Z}_m$ with respect to the Hodge metric. Then $\mathcal{Z}_{m}^H$ is the smallest complete space with respect to the Hodge metric that contains $\mathcal{Z}_m$. Although the compact space $\bar{\mathcal{Z}}_{m}$ may not be unique, the Hodge metric completion space $\mathcal{Z}^H_{m}$ is unique up to isometry. In particular, $\mathcal{Z}^H_{m}\subseteq\bar{\mathcal{Z}}_{m}$ and thus the complex codimension of the complement $\mathcal{Z}^H_{m}\backslash \mathcal{Z}_m$ is at least one. By Lemma 2.7 in \cite{CGL}, we conclude that the mectic completion $\mathcal{Z}^H_{m}$ is a dense and open smooth submanifold in $\bar{\mathcal{Z}}_m$ with $\text{codim}_{\mathbb{C}}(\mathcal{Z}^H_{m}\backslash \mathcal{Z}_m)\geq 1$, and $\mathcal{Z}^H_{m}\backslash \mathcal{Z}_m$ consists of those points in $\bar{\mathcal{Z}}_{m}$ around which the so-called Picard-Lefschetz transformations are trivial.
Moreover the extended period map $\Phi^H_{_{\mathcal{Z}_m}}:\, \mathcal{Z}^H_{m}\to D/\Gamma$ is proper and holomorphic.
Notice that the proof of this conclusion uses extension of the period map as wells as the basic definitions of metric completion. Thus $\mathcal{T}^H_{m}$ is a connected and simply connected complete smooth complex manifold.
We also obtain the following commutative diagram:
\begin{align*}
\xymatrix{\mathcal{T}\ar[r]^{i_{m}}\ar[d]^{\pi_m}&\mathcal{T}^H_{m}\ar[d]^{\pi_{m}^H}\ar[r]^{{\Phi}^{H}_{m}}&D\ar[d]^{\pi_D}\\
\mathcal{Z}_m\ar[r]^{i}&\mathcal{Z}^H_{m}\ar[r]^{{\Phi}_{_{\mathcal{Z}_m}}^H}&D/\Gamma,
}
\end{align*}
where $i$ is the inclusion map, $i_m$ is a lifting of $i\circ \pi_{m}$, and $\Phi^H_{m}$ is a lifting of $\Phi^H_{_{\mathcal{Z}_m}}\circ \pi_{m}^H$, and we fix a suitable choice of $i_m$ and $\Phi^H_{m}$ such that $\Phi=\Phi^H_{m}\circ i_m$. Let us denote $\mathcal{T}_m:=i_{m}(\mathcal{T})$ and the restriction map $\Phi_m=\Phi^H_{m}|_{\mathcal{T}_m}$, then we also have $\Phi=\Phi_m\circ i_m$. Moreover, it is not hard to show that $\Phi_m$ is also bounded by Theorem \ref{bound}. With these notations, we prove that the image $\mathcal{T}_m$ equals to the preimage $(\pi^H_{m})^{-1}(\mathcal{Z}_{m})$. Therefore, $\mathcal{T}_m$ is a connected open complex submanifold in $\mathcal{T}^H_{m}$ and $\text{codim}_{\mathbb{C}}(\mathcal{T}^H_{m}\backslash\mathcal{T}_m)\geq 1$. It is easy to see that $\mathcal{Z}_{m}^{H}\setminus \mathcal{Z}_{m}$ is an analytic subvariety of $\mathcal{Z}_{m}^{H}$, and hence the set $\mathcal{T}^H_{m}\backslash\mathcal{T}_m$ is also an analytic subvariety of $\mathcal{T}^H_{m}$.

Recall that in $\S$\ref{affine structure}, we have fixed a base point $p\in\mathcal{T}$ and an adapted basis $\{\eta_0, \cdots, \eta_{m-1}\}$ for the Hodge decomposition of the base point $\Phi(p)\in D$. With the fixed base point in $D$, we can identify $N_+$ with its unipotent orbit in $\check{D}$.  Then applying the Riemann extension theorem to the bounded map $\Phi_m: \mathcal{T}_m\rightarrow N_+\cap D$, we obtain the following lemma.
\begin{lemma}\label{Riemannextension}
The map $\Phi^{H}_{m}$ is a bounded holomorphic map from $\mathcal{T}^H_{m}$ to $N_+\cap D$.
\end{lemma}

Now we consider the composite,
$$\Psi^{H}_{m} :\, \T^{H}_{m} \to A\cap D,$$
such that $\Psi^{H}_{m}=P\circ \Phi^{H}_{m}$, where $P$ is the projection map from $N_{+}\cap D$ to its subspace $A\cap D$.
Moreover, we also have $\Psi=P\circ\Phi=P\circ \Phi^H_{m}\circ i_m=\Psi^H_{m}\circ i_m$. Then by using Hodge bundles and the property that the holomorphic map $\Psi$ defines the holomorphic affine structure on $\mathcal{T}$, we prove the following theorem.

\begin{theorem}\label{THmaffine}The holomorphic map $\Psi^H_{m}: \,\mathcal{T}^H_{m}\rightarrow A\cap D$ is a local embedding, therefore it defines a global holomorphic affine structure on $\mathcal{T}^H_{m}$. \end{theorem}

Now by the completeness of $\mathcal{T}^H_{m}$ with Hodge metric,  $\Psi^H_{m}$ is an isometry with Hodge metric on $\mathcal{T}^H_{m}$, and a result of Griffiths and Wolf \cite{GW}, we can conclude that the holomorphic map $\Psi^H_{m}: \,\mathcal{T}^H_{m}\rightarrow A\cap D$ is a covering map.
Using the completeness of $\mathcal{T}^H_{m}$ and the holomorphic affine structure on it, we can show that any two points in $\mathcal{T}^H_{m}$ can be joined by a straight line segment. Then the injectivity of $\Psi^H_{m}$ follows by contradiction. In fact,  if $\Psi^H_{m}(p)=\Psi^H_{m}(q)$ for some $p\neq q$, then the fact that $\Psi^H_{m}$ is an affine map would imply that $\Psi^H_{m}$ is a constant map on the line connecting $p$ and $q$. However, this contradicts to the local injectivity of $\Psi^H_{m}$. Moreover, as $\Psi^H_{m}=P\circ \Phi^{H}_{m}$, where $P$ is the projection map and $\Phi^{H}_{m}$ is a bounded map, we may conclude that $\Psi^H_{m}$ is bounded as well. To conclude, we have the following theorem, and one may refer to Section 4.3 in \cite{CGL} for its detailed proof. 
\begin{theorem}\label{injectivityofPhiH}For any $m\geq 3$, the holomorphic map $\Psi^H_{m}:\, \mathcal{T}^H_{m}\to A\cap D$ is an injection and hence bi-holomorphic. In particular, the completion space $\mathcal{T}^H_{m}$ is a bounded domain $A\cap D$ in $A\simeq \mathbb{C}^N$. Moreover, the holomorphic map $\Phi^H_{m}: \,\mathcal{T}^H_{m}\rightarrow N_+\cap D$ is also an injection.
\end{theorem}
%Since the holomorphic affine map $\Psi^H_{m}=P\circ \Phi^H_{m}: \,\mathcal{T}^H_{m}\rightarrow \mathbb{C}^N$ is injective, where $P$ is a projection map and $\Phi^H_{m}$ is a bounded map, we have the following corollaries.
%\begin{corollary}\label{embedTHmintoCN}The holomorphic map $\Phi^H_{m}: \,\mathcal{T}^H_{m}\rightarrow N_+\cap D$ is also an injection.\end{corollary}
%This corollary follows directly from that $\Psi^H_{m}=P\circ \Phi^H_{m}$ with $P$ the projection map and $\Psi^H_{m}$ injective.

The above corollary implies directly that the definition of $\mathcal{T}^H_{m}$ is independent of the choice of the level $m$ structure.
\begin{proposition}\label{indepofm} For any $m, m'\geq 3$,
the complete complex manifolds $\mathcal{T}^H_m$ and $\mathcal{T}^H_{m'}$ are biholomorphic to each other.
\end{proposition}
%To justify this proposition, note that we have defined $\mathcal{T}_m=i_m(\mathcal{T})$, $\mathcal{T}_{m'}=i_{m'}(\mathcal{T})$ and $\Phi_m=\Phi^H_{m}|_{\mathcal{T}_m}$, $\Phi_{m'}=\Phi^H_{m'}|_{\mathcal{T}_{m'}}$. Because $\Phi^H_m$ and $\Phi^H_{m'}$ are embeddings, $\mathcal{T}_m\cong \Phi^H_{m}(\mathcal{T}_m)=\Phi^H_{m}(i_m(\mathcal{T}))$ and $\mathcal{T}_{m'}\cong\Phi^H_{m}(\mathcal{T}_m)=\Phi^H_{m'}(i_{m'}(\mathcal{T}))$. Since the composition maps $\Phi^H_{m}\circ i_m=\Phi$ and $\Phi^H_{m'}\circ i_{m'}=\Phi$, we get $\Phi^H_{m}(i_m(\mathcal{T}))=\Phi(\mathcal{T})=\Phi^H_{m'}(i_{m'}(\mathcal{T}))$. Then $\mathcal{T}_{m}\cong\Phi(\mathcal{T})\cong\mathcal{T}_{m'}$ biholomorphically. Moreover, as $\mathcal{T}^H_m$ and $\mathcal{T}^H_{m'}$ are Hodge metric completion spaces of $\mathcal{T}_m$ and $\mathcal{T}_{m'}$, respectively, the uniqueness of the metric completion spaces implies that $\mathcal{T}^H_m$ is biholomorphic to $\mathcal{T}^H_{m'}$.

%Therefore, we can also give the following definitions and then Theorem \ref{defthm} follows directly by the definition.
This allows us to introduce simplified notations.

 \begin{definition}We define the complete complex manifold $\mathcal{T}^H=\mathcal{T}^H_{m}$, the holomorphic map $i_{\mathcal{T}}: \,\mathcal{T}\to \mathcal{T}^H$ by $i_{\mathcal{T}}=i_m$, and the extended period map $\Phi^H:\, \mathcal{T}^H\rightarrow D$ by $\Phi^H=\Phi^H_{m}$ for any $m\geq 3$.
In particular, with these new notations, we have the commutative diagram
\begin{align*}
\xymatrix{\mathcal{T}\ar[r]^{i_{\mathcal{T}}}\ar[d]^{\pi_m}&\mathcal{T}^H\ar[d]^{\pi^H_{m}}\ar[r]^{{\Phi}^{H}}&D\ar[d]^{\pi_D}\\
\mathcal{Z}_m\ar[r]^{i}&\mathcal{Z}^H_{m}\ar[r]^{{\Phi}_{_{\mathcal{Z}_m}}^H}&D/\Gamma.
}
\end{align*}
\end{definition}
%\begin{remark}\label{defthm}By the definition, we can immediately conclude that the complex manifold $\mathcal{T}^H$ is a complex affine manifold, which is a bounded domain in $\mathbb{C}^N$.\end{remark}

\begin{proposition}\label{i_T}
Let $\T_{0}\subset \T^{H}$ be defined by $\T_{0}:=i_{\T}(\T)$. Then $\T_{0}$ is biholomorphic to the Torelli space $\T'$.
\end{proposition}

The proof of Proposition \ref{i_T} is equivalent to constructing a bi-holomorphic map $\pi_{0}:\, \T_{0}\to \T'$, which fits into the following commutative diagram
$$\xymatrix{
\T \ar[dr]^-{\pi}\ar[rr]^-{i_{\T}} \ar[dd]^-{\pi_m} &&\T_{0}\ar[dl]^-{\pi_{0}}  \ar[dd]^-{\pi_{m}^{H}|_{\T_{0}}}\ar[rr]^-{\Phi^{H}|_{\T_{0}}}  &&D\ar[dd]^-{\pi_{D}}\\
&\T'\ar[dl]_-{\pi'_{m}}\ar[urrr]^-{\Phi'}&&&\\
\mathcal{Z}_{m} \ar[rr]^-{i_{m}} &&\mathcal{Z}_{m}^{H} \ar[rr]^-{\Phi_{\mathcal{Z}_{m}^{H}}} &&D/\Gamma .}$$
Here the markings of the Calabi--Yau manifolds come into play substantially.

From Proposition \ref{i_T}, we can see that the complete complex manifold $\mathcal{T}^H$ is actually the completion space of the Torelli space $\T'$ with respect to the Hodge metric.

Since the restriction map $\Phi^{H}|_{\T_{0}}$ is injective and $\Phi'=\Phi^{H}|_{\T_{0}}\circ (\pi_{0})^{-1}$, we get the global Torelli theorem for the period map $\Phi':\, \T'\to D$ from the Torelli space to the period domain as follows.
\begin{theorem}[Global Torelli theorem]
The period map $\Phi':\, \mathcal{T}'\rightarrow D$ is injective.
\end{theorem}
Using the completeness of $\mathcal{T}^H$ and the injectivity of $\Phi^H$, together with the function $f: \, D\to \mathbb{R}$ which is constructed in Theorem 8.1 in \cite{GS}, we can construct a plurisubharmornic exhaustion function $\mathcal{T}^H$. This shows that $\mathcal{T}^H$ is a bounded domain of holomorphy in $\mathbb{C}^N$. Moreover, the existence of the K\"ahler-Einstein metric follows directly from a theorem of Mok--Yau in \cite{MokYau}.
\begin{theorem}The Hodge metric completion space $\mathcal{T}^H$ of the Torelli space $\T'$ is a bounded domain of holomorphy in $A\simeq \mathbb{C}^N$; thus there exists a complete K\"ahler--Einstein metric on $\mathcal{T}^H$.
\end{theorem}
We remark that we also proved such global Torelli theorem for polarized and marked Calabi--Yau type manifolds in \cite{CGL2}.

\section{Applications}\label{App}

This section contains several applications of the results reviewed in the previous sections.
In Section \ref{surjectivity} we prove a general result for the extended period map to be a biholomorphic map from $\mathcal{T}^H$, the Hodge metric completion of the Torelli space $\mathcal{T}'$ of Calabi--Yau manifolds to the corresponding period domain; and apply this result to the cases of K$3$ surfaces and cubic fourfolds. Such surjectivity results of the period maps may be obtained for the case of polarized and marked hyperk\"ahler manifolds with weight 2 primitive cohomology, as well as certain refined period maps on cubic threefolds. In Section \ref{globalSection} we construct explicit holomorphic sections of the Hodge bundles on $\mathcal{T}^H$, which trivialize those Hodge bundles. In particular, for Calabi--Yau manifolds, a global holomorphic section of holomorphic $(n, 0)$-classes on $\mathcal{T}^H$ is constructed, which coincides with explicit local Taylor expansion in the affine coordinates at any base point $p$ in $\mathcal{T}^H$. Finally in Section \ref{split} we prove a global splitting property for the Hodge bundles, as well as a theorem asserting that a compact polarized and marked Calabi--Yau manifold with complex structure $J$ can not be deformation equivalent to a polarized and marked Calabi--Yau manifolds with conjugate complex structure $-J$. We remark that the same results in Sections 4.2 and 4.3 also hold on the Teichm\"uller space with the same proofs.

\subsection{Surjectivity of the period map on the Hodge metric completion space}\label{surjectivity}In this section we use our results on the Hodge metric completion space $\mathcal{T}^H$ to give a simple proof of the surjectivity of the period maps of K$3$ surfaces and cubic fourfolds. First we have the following general result for polarized and marked Calabi--Yau manifolds and Calabi--Yau type manifolds.
%In fact, based on the fact that $\mathcal{T}^H$ is a complete smooth manifold with respect to Hodge metric, in the cases when $\mathcal{T}$ is a Teichm\"uller space of Calabi--Yau or Calabi--Yau type manifolds and the dimension of $\mathcal{T}^H$ is equal to the dimension of period domain $D$, we have the following conclusion,
\begin{theorem}\label{surj}If $\dim\mathcal{T}^H=\dim D$, then the extended period map $\Phi^H: \mathcal{T}^H\to D$ is surjective on $\mathcal{T}^H$.
\end{theorem}
\begin{proof}
Since $\dim \mathcal{T}^H=\dim D$, the property that $\Phi^H:\,\mathcal{T}^H\rightarrow D$ is an local isomorphism shows that the image of $\mathcal{T}^H$ under the extended period map $\Phi^H$ is open in $D$. On the other hand, the completeness of $\mathcal{T}^H$ with respect to Hodge metric implies that the image of $\mathcal{T}^H$ under $\Phi^H$ is closed in $D$. As $\mathcal{T}^H$ is not empty and that $D$ is connected, we can conclude that $\Phi^H(\mathcal{T}^H)=D$.
\end{proof}
%The main ingredients we use for this theorem is that the condition that $\dim\mathcal{T}^H=\dim D$, the injectivity of the extended period map, and the completeness of $\mathcal{T}^H$ as well as the fact that $D$ is connected. One may refer to \cite[Theorem 4.1]{CGL3} for a detailed proof. 
%\begin{proof}since $\dim \mathcal{T}^H=\dim D$, the property that $\Phi^H:\,\mathcal{T}^H\rightarrow D$ is an local isomorphism shows that the image of $\mathcal{T}^H$ under the extended period map $\Phi^H$ is open in $D$. On the other hand, the completeness of $\mathcal{T}^H$ with respect to Hodge metric implies that the image of $\mathcal{T}^H$ under $\Phi^H$ is close in $D$. As $\mathcal{T}^H$ is not empty and that $D$ is connected, we can conclude that $\Phi^H(\mathcal{T}^H)=D$.\end{proof}
It is well known that for K$3$ surfaces, which are two dimensional Calabi--Yau manifolds, we have $\dim\mathcal{T}^H=\dim\mathcal{T}=\dim D=19$; for cubic fourfolds, they are Calabi--Yau type manifolds. One knows that both K$3$ and cubic fourfolds have smooth moduli spaces with level $m$ structure for big enough $m$, and $\dim\mathcal{T}^H=\dim\mathcal{T}=\dim D=20$. Thus Theorem \ref{surj} can be applied to conclude that that the extended period from the Hodge metric completion of Torelli space to the period domain is surjective for polarized and marked K3 surfaces or cubic fourfolds. 

Let $\mathcal{T}'$ be the Torelli space of hyperk\"ahler manifolds and $H^2_{pr}(M,\mathbb{C})$ the degree $2$ primitive cohomology group. Let $D$ the period domain of weight two Hodge structures on $H^2_{pr}(M, \mathbb{C})$. Then our method in \cite{CGL} can be applied to prove that the period map from $\mathcal{T}'$ and $D$ is injective. Moreover, we also know that the Hodge completion space $\mathcal{T}^H$ of the Torelli space has the same dimension as the period domain $D$ of weight two Hodge structures on $H^2_{pr}(M, \mathbb{C})$. Thus Theorem \ref{surj} can also applied to conclude that the extended period map from $\mathcal{T}^H$ to $D$ is surjective. One may refer to \cite{ver} and \cite{Huy} for different injectivity and surjectivity results for hyperk\"ahler manifolds. To conclude we have the the following corollary. 

\begin{corollary}
Let $\mathcal{T}^H$ be the Hodge metric completion space of the Torelli space for K$3$ surfaces, cubic fourfolds, or hyperk\"ahler manifolds. Then the extended period map $\Phi^H: \,\mathcal{T}^H\rightarrow D$ is surjective.
\end{corollary}
%In fact, among all the Calabi--Yau or Calabi--Yau type projective hypersurfaces, K$3$ surfaces and cubic fourfolds are the only types satisfying the condition that the dimensions of the Teichm\"uller space and the period domain are the same. 
%\begin{remark}
%Let $\mathcal{T}$ be the Teichm\"uller space of polarized and marked hyperk\"ahler manifolds, $H^2_{pr}(M,\mathbb{C})$ the degree $2$ primitive cohomology group, and $D$ the period domain of weight two Hodge structures on $H^2_{pr}(M,\mathbb{C})$. Then our method can be applied without change to prove that the period map from $\mathcal{T}$ to $D$ is also injective. Furthermore, let $\mathcal{T}^H$ be the Hodge completion of $\mathcal{T}$ with respect to the Hodge metric induced from the homogeneous metric on $D$, then the extended period map from $\mathcal{T}^H$ to $D$ is also surjective. This follows from the same argument of above theorem. See \cite{ver} and \cite{Huy} for different injectivity and surjectivity results for hyperk\"ahler manifolds.
%\end{remark}
\subsection{Global holomorphic sections of the Hodge bundles}\label{globalSection}
In this section we prove the existence and study the property of global holomorphic sections of the Hodge bundles $\{F^k\}_{k=0}^n$ over the Hodge metric completion space $\mathcal{T}^H$ of the Torelli space of Calabi--Yau manifolds. Same results hold on the Techm\"uller space by pulling back through the covering map $i_\T:\, \T\to \T^H$.

Recall that we have fixed a base point $p\in\mathcal{T}^{H}$ and an adapted basis $\{\eta_0, \cdots, \eta_{m-1}\}$ for the Hodge decomposition of the base point $\Phi^{H}(p)\in D$. With the fixed base point in $D$, we can identify $N_+$ with its unipotent orbit in $\check{D}$ by identifying an element $c\in N_+$ with $[c]=cB$ in $\check{D}$.
On one hand, as we have fixed an adapted basis $\{\eta_0, \cdots, \eta_{m-1}\}$ for the Hodge decomposition of the base point. Then elements in $G_{\mathbb{C}}$ can be identified with a subset of the nonsingular block matrices. In particular, the set $N_+$ is identified with its unipotent orbit in $\check{D}$. Then elements in $N_+$ can be realized as nonsingular block lower triangular matrices whose diagonal blocks are all identity submatrix. Namely, for any element $\{F^k_o\}_{k=0}^n\in N_+\subseteq\check{D}$, there exists a unique nonsingular block lower triangular matrices $A(o)\in G_{\mathbb{C}}$ such that
$(\eta_0, \cdots, \eta_{m-1})A(o)$ is an adapted basis for the Hodge filtration $\{F^k_o\}\in N_+$ that represents this element in $N_+$. Similarly, any elements in $B$ can be realized as nonsingular block upper triangular matrices in $G_{\mathbb{C}}$. Moreover, as $\check{D}=G_{\mathbb{C}}/B$, we have that for any $U\in G_{\mathbb{C}}$, which is a nonsingular block upper triangular matrix, $(\eta_0, \cdots, \eta_{m-1})A(o)U$ is also an adapted basis for the Hodge filtration $\{F^k(o)\}_{k=0}^n$. Conversely, if $(\zeta_0, \cdots, \zeta_{m-1})$ is an adapted basis for the Hodge filtration $\{F^k_o\}_{k=0}^n$, then there exists a unique $U\in G_{\mathbb{C}}$ such that $(\zeta_0, \cdots, \zeta_{m-1})=(\eta_0, \cdots, \eta_{m-1})A(o)U$. For any $q\in\mathcal{T}^{H}$, let us denote the Hodge filtration at $q\in \mathcal{T}^{H}$ by $\{F_q^k\}_{k=0}^n$, then we have that $\{F_q^k\}_{k=0}^n\in N_+\cap D$ by Theorem \ref{injectivityofPhiH}. Thus there exists a unique nonsingular block lower triangular matrices $\tilde{A}(q)$ such that $(\eta_0, \cdots, \eta_{m-1})\tilde{A}(q)$ is an adapted basis for the Hodge filtration $\{F^k_q\}_{k=0}^{n}$.

On the other hand, for any adapted basis $\{\zeta_0(q), \cdots, \zeta_{m-1}(q)\}$ for the Hodge filtration $\{F^k_q\}_{k=0}^n$ at $q$, we know that there exists an $m\times m$ transition matrix $A(q)$ such that
%\begin{align*}
$(\zeta_0(q), \cdots, \zeta_{m-1}(q))=(\eta_0, \cdots, \eta_{m-1})A(q).$
%\end{align*}
Moreover, we set the blocks of $A(q)$ as in \eqref{block} and denote the ${(i,j)}$-th block of $A(q)$ by $A^{i,j}(q)$.

As both $(\eta_0, \cdots, \eta_{m-1})\tilde{A}(q)$ and $(\eta_0, \cdots, \eta_{m-1})A(q)$ are adapted bases for the Hodge filtration for $\{F^k_q\}_{k=0}^n$, there exists a $U\in G_{\mathbb{C}}$ which is a block nonsingular upper triangular matrix such that $(\eta_0, \cdots, \eta_{m-1})\tilde{A}(q)U=(\eta_0, \cdots, \eta_{m-1})A(q).$
Therefore, we conclude that
\begin{align}\label{upper lower}\tilde{A}(q)U=A(q).
\end{align}
where $\tilde{A}(q)$ is a nonsingular block lower triangular matrix in $G_{\mathbb{C}}$ with all the diagonal blocks equal to identity submatrix, while $U$ is a block upper triangular matrix in $G_{\mathbb{C}}$. However, according to basic linear algebra, we know that a nonsingular matrix $A(q)\in G_{\mathbb{C}}$ have the decomposition of the type in \eqref{upper lower} if and only if the principal submatrices $[A^{i,j}(q)]_{0\leq i,j\leq n-k}$ are nonsingular for all $0\leq k\leq n$.

To conclude, by Theorem \ref{injectivityofPhiH}, we have that $\Phi^{H}(q)\in N_+$ for any $q\in \mathcal{T}^{H}$. Therefore, for any adapted basis $(\zeta_0(q), \cdots, \zeta_{m-1}(q))$, there exists a nonsingular block matrix $A(q)\in G_{\mathbb{C}}$ with $\det[A^{i,j}(q)]_{0\leq i,j\leq n-k}\neq 0$ for any $0\leq k\leq n$ such that
%\begin{align*}
$(\zeta_0(q), \cdots, \zeta_{m-1}(q))=(\eta_0, \cdots, \eta_{m-1})A(q).$
%\end{align*}
Let $\{F^k_p\}_{k=0}^n$ be the reference Hodge filtration at the base point $p\in\mathcal{T}^{H}$.
For any point $q \in \mathcal{T}^H$ with the corresponding Hodge filtrations $\{F^k_q\}_{k=0}^n$, we define the following maps
\begin{align*}
P^k_q: \, F^k_q\to F^k_p \quad\text{for any}\quad 0\leq k\leq n
\end{align*}
to be the projection map with respect to the Hodge decomposition at the base point $p$. With the above notation, we therefore have the following lemma.  
\begin{lemma}\label{globalTransversality}
For any point $q\in\mathcal{T}^H$ and $0\leq k\leq n$, the map $P_q^k:\, F^k_q\to F^k_p$ is an isomorphism. Furthermore, $P_q^k$ depends on $q$ holomorphically.
\end{lemma}
\begin{proof}
We have already fixed $\{\eta_0, \cdots, \eta_{m-1}\}$ as an adapted basis for the Hodge decomposition of the Hodge structure at the base point $p$. Thus it is also the adapted basis for the Hodge filtration $\{F^k_p\}_{k=0}^n$ at the base point.
For any point $q\in \mathcal{T}^{H}$, let $\{\zeta_0, \cdots, \zeta_{m-1}\}$ be an adapted basis for the Hodge filtration $\{F^k_q\}_{k=0}^n$ at $q$. Let $[A^{i,j}(q)]_{0\leq i,j\leq n}\in G_{\mathbb{C}}$ be the transition matrix between the basis $\{\eta_0,\cdots, \eta_{m-1}\}$ and $\{\zeta_0, \cdots, \zeta_{m-1}\}$ for the same vector space $H^{n}(M, \mathbb{C})$. We have showed that $[A^{i,j}(q)]_{0\leq i,j\leq n-k}$ is nonsingular for all $0\leq k\leq n$.

On the other hand, the submatrix $[A^{i,j}(q)]_{0\leq j\leq n-k}$ is the transition matrix between the bases of $F^k_q$ and $F^0_p$ for any $0\leq k\leq n$, that is
\begin{align*}
(\zeta_0(q), \cdots, \zeta_{f^k-1}(q))=(\eta_0, \cdots, \eta_{m-1})[A^{i,j}(q)]_{0\leq j\leq n-k} \quad\text{for any}\quad 0\leq k\leq n,
\end{align*}
where $(\zeta_0(q), \cdots, \zeta_{f^k-1}(q))$ and $(\eta_0, \cdots, \eta_{m-1})$ are the bases for $F^k_q$ and $F^0_p$ respectively.
Thus the matrix of $P_q^k$ with respect to $\{\eta_0, \cdots, \eta_{f^k-1}\}$ and $\{\zeta_0, \cdots, \zeta_{f^k-1}\}$ is the first $(n-k+1)\times (n-k+1)$ principal submatrix $[A^{i,j}(q)]_{0\leq i,j\leq {n-k}}$ of $[A^{i,j}(q)]_{0\leq i,j\leq n}$. Now since $ [A^{i,j}(q)]_{0\leq i,j\leq {n-k}}$ for any $0\leq k\leq n$ is nonsingular, we conclude that the map $P_q^k$ is an isomorphism for any $0\leq k\leq n$.

From our construction, it is clear that the projection $P_q^k$ depends on $q$ holomorphically.

\end{proof}

Using this lemma, we are ready to construct the global holomorphic sections of Hodge bundles over $\mathcal{T}^H$. 
For any $0\leq k\leq n$, we know that $\{\eta_0, \eta_1, \cdots, \eta_{f^{k}-1}\}$ is an adapted basis of the Hodge decomposition of $F^k_p$ for any $0\leq k\leq n$. Then we define the sections
\begin{align}\label{sections}
s_i: \mathcal{T}^H\rightarrow F^{k}, \quad q\mapsto (P_q^k)^{-1}(\eta_i)\in F^k_q\quad\text{for any}\quad 0\leq i\leq f^{k}-1.
\end{align}
%
%then we define the sections of $F^k$ over $\mathcal{T}^H$ with clear geometrical meaning,
%\begin{align}\label{sections}
%(P_q^k)^{-1}(\eta_i)\quad\text{for any }q\in\mathcal{T}^H \text{ and } 0\leq i\leq f^{k}-1.
%\end{align}
Lemma \ref{globalTransversality} implies that $\{(P_q^k)^{-1}(\eta_i)\}_{i=0}^{f^k-1}$ form a basis of $F^k_q$ for any $q\in \mathcal{T}^H$. In fact, we have proved the following theorem for polarized and marked Calabi--Yau manifolds.
\begin{theorem}\label{trivial bundle}
For all $0\leq k\leq n$, the Hodge bundles $F^k$ over $\mathcal{T}^H$ are trivial bundles, and the trivialization can be obtained by $\{s_i\}_{0\leq i\leq f^k-1}$ which is defined in \eqref{sections}. In particular, the section $s_0: \mathcal{T}^H\to F^n$ is a global nowhere zero section of the Hodge bundle $F^n$ for Calabi--Yau manifolds.
\end{theorem}
%\begin{remark}
%In particular, for the Hodge bundle $F^n$ over $\mathcal{T}^H$, the global section $\phi_0: \,\mathcal{T}^H\rightarrow F^n$ given by $\phi_0(q)=(P_q^n)^{-1}(\eta_0)$ is a holomorphic extension of the canonical family of $(n,0)$-class over the Kuranish coordinate chart around reference point $p\in \mathcal{T}^H$.
%\end{remark}

With the adapted basis at the base point $p\in \T^{H}$, we can also see $\Phi_*(\text{T}^{1,0}_{p}(\T))=\mathfrak{a}\subset \mathfrak{n}_{+}$ as a block lower triangle matrix whose diagonal elements are zero. Moreover by local Torelli theorem for Calabi--Yau manifolds, we can conclude that $\mathfrak{a}$ is isomorphic to its $(1,0)$-block as vector spaces, see \eqref{block} for the definition. Let $(\tau_{1},\cdots ,\tau_{N})^{T}$ be the $(1,0)$-block of $\mathfrak{a}$. 
Sine the affine structure on $A$ is induced by $\exp:\, \mathfrak{a}\to A$ which is an isomorphism, $(\tau_{1},\cdots ,\tau_{N})^{T}$ also defines a global affine structure on $A$, and hence on $\T^{H}$. We denote it by
$$\tau^{H}:\, \T^{H} \to \C^{N},\quad  q\mapsto (\tau_{1}(q),\cdots ,\tau_{N}(q)).$$
Note that from linear algebra, it is easy to see that the $(1,0)$-block of $A=\exp(\mathfrak{a})$ is still $(\tau_{1},\cdots ,\tau_{N})^{T}$. Hence the affine map defined as above can be constructed as the $(1,0)$-block of the image of the period map.
To be precise, let $P^{1,0}:\, N_{+}\to \C^{N}$  be the projection of the matrices in $N_{+}$ onto their $(1,0)$-blocks. Then the affine map is
$$\tau^{H}=P^{1,0}\circ \Phi^{H}: \, \T^{H} \to \C^{N}.$$
Moreover $\tau^{H}$ is injective, and hence it defines another embedding of $\T^{H}$ into $\C^{N}$.

By using the local deformation theory for Calabi--Yau manifolds in \cite{tod1}, Todorov constructed a canonical local holomorphic section of the line bundle $F^n$ over the local deformation space of a Calabi--Yau manifold. In fact, let $\Omega_p$ be a holomorphic $(n,0)$-form on the central fiber $M_p$ of the family. Then there exists a coordinate chart $\{U_p, (\tau_1,\cdots, \tau_N)\}$ around the base point $p$ and a basis $\{\phi_1,\cdots, \phi_N\}$ of harmonic Beltrami differentials $\mathbb{H}^{0,1}(M_p, T^{1,0}M_p)$, such that
\begin{align}\label{can10}
\Omega^c_p(\tau)=e^{\phi(\tau)}\lrcorner\Omega_p,
\end{align}
is a family of holomorphic $(n,0)$-forms over $U_p$. We can assume this local coordinate chart is the same as the affine coordinates at $p$ we constructed as above, which can be achieved simply by taking $p$ as the base point. The Kuranishi family of Beltrami differentials $\phi(\tau)$ satisfyies the integrability equation $\bar{\partial} \phi(\tau) = \frac{1}{2} [\phi(\tau),\phi(\tau)]$ and the gauge condition $\bar{\partial}^*\phi(\tau)=0$ which is solvable for Calabi--Yau manifolds by the Tian-Todorov lemma, and  the Taylor expansion $\phi(\tau)=\sum_{i=1}^N\phi_i\tau_i+O(|\tau|^2)$ converges for $|\tau|$ small by classical Kodaira-Spencer theory. Then one may conclude the following lemma, and the detailed proof of the following lemma can be found in \cite[pp.12--14]{CGL}, or \cite[Proposition 5.1]{LRY}.  
\begin{lemma}\label{expcoh}
Let $\Omega^c_p(\tau)$ be a canonical family defined by
\eqref{can10}.
Then we have the following section of $F^n$ over $U_p$,
\begin{eqnarray}\label{cohexp10}
[\Omega^c_p(\tau)]=[\Omega_p]+\sum_{i=1}^N
\tau_i[\phi_i\lrcorner\Omega_p]+A(\tau),
\end{eqnarray}
where $\{[\phi_i\lrcorner\Omega_p]\}_{i=1}^N$ give a basis of $H^{n-1,1}(M_p)$, and $$A(\tau)=O(|\tau|^2)\in\bigoplus_{k=2}^n H^{n-k,k}(M_p)$$ denotes terms of
order at least $2$ in $\tau$.
\end{lemma}
%\begin{proof}
%Details of the proof of this lemma can be found in  In fact one can directly check the following formula
%\begin{align}\label{si}
%e^{- \phi(\tau)}\lrcorner \left( d  (e^{\phi(\tau)}\lrcorner \Omega_p)\right)=\bar{\partial}\Omega_p+\p(\phi(\tau)\lrcorner
%\Omega_p).
%\end{align}
%The construction of the Kuranishi family $\phi(\tau)$ implies that $\partial(\phi(\tau)\lrcorner \Omega_p)=0$, and the fact that $\Omega_p$ is holomorphic on $M_p$ implies $\bar{\partial}\Omega_p=0$ . So the right hand side of formula \eqref{si} is equal to $0$.
%Then by replacing the de Rham differential operator $d$ on the left hand side by $\partial_\tau +\bar{\partial}_\tau$ on fiber $M_\tau$, we get
%$(\partial_\tau+\bar{\partial}_\tau)(e^{\phi(\tau)}\lrcorner \Omega_p)=0.$
%Note that $e^{\phi(\tau)}\lrcorner \Omega_p$ is a $(n, 0)$ form on $M_\tau$, and $\partial_\tau(e^{\phi(\tau)}\lrcorner \Omega_p)=0$, we get
%$$\bar{\partial}_\tau(e^{\phi(\tau)}\lrcorner \Omega_p)=0.$$
%Therefore $e^{\phi(\tau)}\lrcorner \Omega_p$ is a holomorphic $(n, 0)$- form on the Calabi--Yau manifold $M_\tau$. The Taylor expansion \eqref{cohexp10} follows from the corresponding Taylor expansion of $\phi(\tau)$.
%\end{proof}

Using the same notation as in Lemma \ref{expcoh}, we are ready to prove the following theorem for Calabi--Yau manifolds,
\begin{theorem}\label{globalExp}
Choose $[\Omega_p]=\eta_0$, then the section $s_0$ of $F^n$ is a global holomorphic extension of the local section $[\Omega^c_p(\tau)]$.
\end{theorem}
\begin{proof}
Because both $s_0$ and $[\Omega^c_p(\tau)]$ are holomorphic sections of $F^n$, we only need to show that $s_0|_{U_p}=[\Omega^c_p(\tau)]$. In fact, from the expansion formula \eqref{cohexp10}, we have that for any $q\in U_p$
\begin{align*}
P^n_q([\Omega^c_p(\tau(q))])=[\Omega_p]=\eta_0.
\end{align*}
Therefore, $[\Omega^c_p(\tau(q))]=(P^n_q)^{-1}(\eta_0)=s_0(q)$ for any point $q\in U_p$.
\end{proof}

%As an example, if we consider the Teichm\"uller space of polarized and marked hyperk\"ahler manifolds and the weight two variation of Hodge structure of hyperk\"abler manifolds, then the Taylor series \eqref{cohexp10} is a finite degree polynomial and converges globally on $\mathcal{T}^H$. More precisely we have the following
\begin{example}
Let $\mathcal{T}^H$ be the Hodge completion of the Torelli space of hyprk\"ahler manifold, and $(\tau_1, \cdots, \tau_{N})$ be global affine coordinates with respect to the reference point $p$ and an orthonormal basis $\{\eta_1, \cdots, \eta_{N}\}$ of $H_{pr}^{1,1}(M_p)$, then
\begin{align*}
[\Omega_p^c(\tau)]=[\Omega_p]+\sum\limits_{i=1}^{19}\tau_i\eta_i+\left( \frac{1}{2}\sum\limits_{i=1}^{19}\tau_i^2 \right)[\overline{\Omega}_p],
\end{align*}
is a global holomorphic section of $F^2$ over $\mathcal{T}^H$. In fact, in this case, $\mathcal{T}^H$ is bi-holomorphic to $D$ given by the period map as discussed in Section 4.1, and the affine structure on $\mathcal{T}^H$ is induced from the affine structure on $D$ by the Harish-Chandra embedding of $D$ into the complex Euclidean space. The global affine coordinates on $\mathcal{T}^H$ is induced by the Harish-Chandra embedding.
\end{example}
Note that although we state and prove the results in this section over the completion space $\T^H$, by pulling back through the covering map $i_\T$, we see that these results still hold on the Teichm\"uller space $\T$.

\subsection{A global splitting property of the Hogde bundles}\label{split}
%Since the Hodge bundle $F^0$ is a trivial bundle over $\mathcal{T}^H$, we have a trivialization
%\begin{align*}
%$F^0=\mathcal{T}^H\times H^n(M,\mathbb{C}).$
%\end{align*}
%Then for any sub-bundle $V\subset F^0$, the fiber $V_q$ at a point $q\in \mathcal{T}^H$ will be considered as a subspace of $H^n(M,\mathbb{C})$, which does not depend on the point $q\in\mathcal{T}^H$.

In this section, we will directly construct globally defined anti-holomorphic vector bundles $\tilde{F}^k$ over $\mathcal{T}^H$, such that the vector space $H^n(M,\mathbb{C})$ splits as
\begin{align*}
H^n(M, \mathbb{C})=F^k_p\oplus \tilde{F}^k_q\quad \text{for any } q\in \mathcal{T}^H,
\end{align*}
where $p$ is the base point in $\mathcal{T}^H$. Then as an application, in Theorem \ref{conjugate} we  prove that any two fibers of the universal family $\phi':\,\mathcal{U}'\to\mathcal{T}'$ over the Torelli space $\T'\subset \T^{H}$ can not be complex conjugate manifolds of each other. The construction of vector bundles $\tilde{F}^k$ is again based on Lemma \ref{globalTransversality}. Let us denote $\tilde{F}^k=\bar{{F}^{n-k+1}}$ for each $0\leq k\leq n$. Then we have the following equivalent lemma. 

\begin{lemma}\label{globalTransversality2}
For any $q\in\mathcal{T}^H$ and $1\leq k\leq n$, we have that $H^n(M, \mathbb{C})=F^k_p\oplus \bar{{F}^{n-k+1}_q}$. Moreover, for any different points $q$ and $q'$ on $\mathcal{T}^{H}$, and $1\leq k\leq n$, we have $H^n(M, \mathbb{C})=F^k_q\oplus \bar{{F}^{n-k+1}_{q'}}$.
\end{lemma}
\begin{proof}
Firstly, the decomposition $H^n(M, \mathbb{C})=F^k_p\oplus \bar{{F}^{n-k+1}_p}$ follows from the definition of the Hodge structure for any $0\leq k\leq n$. Secondly Lemma \ref{globalTransversality} implies that $P_q^{n-k+1}: \, F^{n-k+1}_q\to F^{n-k+1}_p$ is an isomorphism for any $q\in\mathcal{T}^H$ and any $0\leq k\leq n$. Therefore $F^k_p\cap \bar{{F}^{n-k+1}_q}=\{0\}$ as the projection from $F^k_p$ to $\bar{F^{n-k+1}_p}$ is a zero map.

On the other hand, $\dim F^k_p + \dim \bar{{F}^{n-k+1}_q}=\dim F^k_p + \dim \bar{{F}^{n-k+1}_p}=\dim H^n(M, \mathbb{C})$, so we have that
$$H^n(M, \mathbb{C})=F^k_p\oplus \bar{{F}^{n-k+1}_q}.$$
\end{proof}
Because the reference point $p$ is an arbitrary prefixed point on $\mathcal{T}'$, and the Hodge filtration at each point does not depend on the choice of the reference point, Lemma \ref{globalTransversality2} actually implies,
\begin{corollary}\label{corOfTransversality}
For any different points $q$ and $q'$ on $\mathcal{T}'$, and $1\leq k\leq n$, we have $H^n(M, \mathbb{C})=F^k_q\oplus \bar{{F}^{n-k+1}_{q'}}$.
\end{corollary}
\begin{theorem}\label{last}
The vector bundles $\{\tilde{F}^k\}_{k=0}^n$ are globally defined anti-holomorphic vector bundles over $\mathcal{T}^H$ such that
%\begin{align*}
$H^n(M, \mathbb{C})=F^k_p\oplus \tilde{F}^k_q$ for any $q\in \mathcal{T}^H.$
%\end{align*}
\end{theorem}

Now we let $M$ be a complex manifold with background differential manifold $X$ and complex structure $J:\, T_{\mathbb{R}}X\to T_{\mathbb{R}}X$, then the complex conjugate manifold $\bar{M}$ is a complex manifold with the same background differential manifold $X$ and with conjugate complex structure $-J$. In fact, $M$ and its complex conjugate manifold $\bar{M}$ satisfy the relation $T^{1,0}M=T^{0,1}\bar{M}$ and $T^{0,1}M=T^{1,0}\bar{M}$. 

Problems regarding deformation inequvalent complex conjugated complex structures have been studied before. For example one may find interesting results in \cite{KK}. We will apply our results to study such problem for polarized and marked Calabi--Yau manifolds. In fact, another interesting application of Lemma \ref{globalTransversality2} is that a polarized and marked Calabi--Yau manifold $M$ can not be connected to its complex conjugate manifold $\bar{M}$ by deformation of complex structure. In fact, for any two points $q$ and $q'$ in the Torelli space $\mathcal{T}'$, let $M_q$ and $M_{q'}$ denote the fibers of the universal family $\phi':\, \mathcal{U}'\to \mathcal{T}'$ over the points $q$ and $q'$ respectively. Then we have the following theorem.

\begin{theorem}\label{conjugate}
If $q$ and $q'$ are two points in the Torelli space $\mathcal{T}'$, then $M_{q'}\neq\bar{M_{q}}$.
\end{theorem}
\begin{proof}
We prove this theorem by contradiction. Suppose $M_{q'}= \bar{M_{q}}$, and let $\Omega$ be an $(n, 0)$ holomorphic form on $M_q$, then $\bar{\Omega}$ is a $(n, 0)$ holomorphic form on $M_{q'}=\bar{M_q}$. Therefore the fibers of Hodge bundles over the two points satisfy $F^{n}_q=\overline{F^n_{q'}}\subset \bar{F^{1}_{q'}}$, therefore
\begin{align*}
H^n(M,\mathbb{C})\neq F^n_q\oplus\bar{F^{1}_{q'}}.
\end{align*}
But this contradicts to Corollary \ref{corOfTransversality}, so $M_{q'}\neq \bar{M_{q}}$ as desired.
\end{proof}
The same result also holds on the Teichm\"uller space.

%%%%%%%%%%%%%%%%%%%%%%%%%%%%%%%%%%%%%%%%%%%%%%%%%%%%%%%%%%%%%%%%%%%%%%%%%%%%%%%%%%%%%%%%%%%%%%%%%%%%%%%%%%%%%%%%%%%%%%%%%%%%%%%%%%%%%%%%%%%%%%%%%%%%%%%%%%%%%%%%%%%%%%%%%%%%%%%%%%%%%%%%%%%%%%%%%%%%%%%%%%%%%%%%%%%

\end{document}